\documentclass[12pt]{amsart}
\usepackage[margin=1.4in]{geometry}
\usepackage[english]{babel}
\usepackage[utf8]{inputenc}
\usepackage{subcaption}
\usepackage{verbatim}
\usepackage{amsmath}
\usepackage{amssymb}
\usepackage{amsfonts}
\usepackage{mathtools}
\usepackage{amsthm}
\usepackage{mathrsfs}
\usepackage[shortalphabetic, abbrev]{amsrefs}
\usepackage[all]{xy}
\usepackage[pdftex]{graphicx}
\usepackage{color}
\usepackage[noadjust]{cite}
\usepackage{url}
\usepackage{indent first}
\usepackage[labelfont=bf,labelsep=period,justification=raggedright]{caption}
\usepackage[english]{babel}
\usepackage[utf8]{inputenc}
\usepackage{hyperref}
\usepackage[colorinlistoftodos]{todonotes}
\usepackage{tkz-fct}
\usepackage{tikz}
\usetikzlibrary{calc}
\usepackage{array}
\usepackage{booktabs}
\DeclareFontFamily{U}{wncy}{}
    \DeclareFontShape{U}{wncy}{m}{n}{<->wncyr10}{}
    \DeclareSymbolFont{mcy}{U}{wncy}{m}{n}
    \DeclareMathSymbol{\Sh}{\mathord}{mcy}{"58}

\DeclareMathOperator{\Aut}{Aut}

\DeclareMathOperator{\Char}{char}

\DeclareMathOperator{\disc}{disc}

\DeclareMathOperator{\GL}{GL}
\DeclareMathOperator{\Hom}{Hom}

\DeclareMathOperator{\ord}{ord}

\DeclareMathOperator{\PGL}{PGL}

\DeclareMathOperator{\red}{red}
\DeclareMathOperator{\Res}{Res}

\DeclareMathOperator{\Spec}{Spec}
\DeclareMathOperator{\Supp}{Supp}

\newcommand{\ri}[1]{\mathfrak{o}_{#1}}

\newcommand{\isom}{\cong}

\newcommand{\QED}{\hspace{\stretch{1}} $\begin{lem}acksquare$}

\newcommand{\CC}{\mathbb{C}}

\newcommand{\FF}{\mathbb{F}}

\newcommand{\PP}{\mathbb{P}}
\newcommand{\QQ}{\mathbb{Q}}

\newcommand{\ZZ}{\mathbb{Z}}

\newcommand{\mcB}{\mathcal{B}}
\newcommand{\mcC}{\mathcal{C}}
\newcommand{\mcE}{\mathcal{E}}
\newcommand{\mcF}{\mathcal{F}}
\newcommand{\mcG}{\mathcal{G}}
\newcommand{\mcH}{\mathcal{H}}
\newcommand{\mcM}{\mathcal{M}}

\newcommand{\mcP}{\mathcal{P}}

\newcommand{\mcR}{\mathcal{R}}

\newcommand{\mff}{\mathfrak{f}}

\newcommand{\Gal}{\mathrm{Gal}}

\newcommand{\lloyd}[1]{\todo[color=blue!30]{LW: #1}}

\theoremstyle{plain}
\newtheorem{thm}{Theorem}
\newtheorem{lem}[thm]{Lemma}
\newtheorem{cor}[thm]{Corollary}

\newtheorem{prop}[thm]{Proposition}

\theoremstyle{definition}
\newtheorem{defn}[thm]{Definition}

\theoremstyle{remark}
\newtheorem{rem}[thm]{Remark}
\newtheorem*{ex}{Example}

\numberwithin{equation}{section}
\numberwithin{thm}{section}

\begin{document}

\title[ Dynamical Shafarevich Theorem ]{A  Dynamical Shafarevich Theorem for Rational Maps over Number Fields and Function Fields}
\author[Szpiro]{Lucien Szpiro}
\address[L. Szpiro]{Department of Mathematics, The Graduate Center, City University of New York; 365 Fifth Avenue,
New York, NY 10016 U.S.A. }
\author[West]{Lloyd West}
\address[L. West]{Department of Mathematics, University of Virginia, Charlottesville, VA 22903 }
\date{\today}

\begin{abstract}
We prove a dynamical Shafarevich theorem on the finiteness of the set of isomorphism classes of rational maps with fixed degeneracies. More precisely, fix an integer $d\geq2$ and let $K$ be either a number field or the function field of a curve $X$ over a field $k$, where $k$ is of characteristic zero or $p>2d-2$ that is either algebraically closed or finite. Let $S$ be a finite set of places of $K$. We prove the finiteness of the set of isomorphism classes of rational maps over $K$ with a natural kind of good reduction outside of $S$. We also prove auxiliary results on finiteness of reduced effective divisors in $\PP^1_K$ with good reduction outside of $S$ and on the existence of global models for rational maps. 
\end{abstract}

\keywords{arithmetic dynamics, Shafarevich theorem, good reduction, moduli}

\subjclass[2010]{14G25, 37P45, 37F10}

\maketitle

 \emph{Dedicated to Joseph Silverman on his $60^{th}$ birthday}

\section{Introduction}

\subsection{Rational maps and dynamical systems}

Let $K$ be a field. By a rational map of degree $d$ over $K$, we mean a separable morphism $f : \PP_K^1 \to \PP_K^1$ defined over $K$. Note that such an $f$ can be thought of as an element of the rational function field $K(z)$; alternatively, a rational map may be written  $f = [f_1(x,y): f_2(x,y)]$ for a pair $f_1,f_2 \in k[x,y]$ of homogeneous polynomials of degree $d$.

The behavior of rational maps under iteration has been intensively studied - both over $\CC$ (see \cite{MilnorBook} for an overview), and more recently over fields of an arithmetic nature, such as $\QQ$, $\QQ_p$ or $\FF_p(t)$ (see \cite{SilvermanBook} for an overview).

The dynamical behavior of a rational map is unchanged after performing the same change of coordinates on the source and target spaces; for example, periodic points are taken to periodic points. This leads to the following notion of isomorphism for rational maps.

\begin{defn}

Let $g,\,f: \PP^1_K \to \PP^1_K$ be two rational maps over $K$. We say that $g $ and $f$ are \emph{$K$-isomorphic}, written $g \isom_K f$, if there exists an automorphism $\gamma \in \Aut(\PP^1_K)$ such that $f = \gamma \circ g \circ \gamma^{-1}$. 

We shall refer to a $K$-isomorphism equivalence class of rational maps defined over $K$ as a \emph{$K$-dynamical system}. We denote the class of a map $f$ by $[f]_K$.   Denote by $\mcM_d^\mathrm{dyn}(K)$ the set of $K$-dynamical systems of degree $d$.  Note that there is a coarse moduli space $M_d^\mathrm{dyn}$ of dynamical systems of degree $d$ (see \cite{SilvermanMod}, \cite{West}). 

Suppose $K$ is the function field of a curve $C$ over a base field $k$. We shall say that a map $f/K$ is \emph{$K$-trivial} if there exists a map $g$ defined over the base field $k$ such that $g_K\isom_K f$, where $g_K$ denotes the base change of $g$ to $K$. We shall say that $f/K$ is \emph{isotrivial} if there exists a map $g$ defined over $\bar{k}$ such that $g_{\bar{K}}\isom_{\bar{K}}f$. 

\end{defn}

\subsection{Models and reduction}\label{ss:ModelsReduction} Let $X$ be a Dedekind scheme. Denote by $K$ the field of functions on $X$. The closed points of $X$ can be identified with the (non-archimedean) places of the field $K$. To a (non-archimedean) place $\mathnormal{v}$ of $K$ associate the valuation ring $\ri{\mathnormal{v}}$ of $K$. Let $S$ be a finite non-empty subset of such places of $X$ and define the ring 
\[
\ri{S} = \bigcap_{\mathnormal{v}\notin S} \ri{\mathnormal{v}}
\] 
Denote the complement, $X\setminus S$, by $X^\circ$; then $X^\circ=\Spec \ri{S}$.

Through the remainder of this paper, we shall take $X$ to be either a complete smooth curve over a base field $k$ or the spectrum of the ring of integers of a number field.

We recall the standard notion of good reduction for dynamical systems.

\begin{defn}

With notations as above, let $f$ be a rational map over $K$ and let $U$ be a scheme with field of fractions $K$; for example, a non-empty open subset of $X$. By a \emph{model for $f$ over $U$} we shall mean a rational map $\mff:\, \PP^1_{U} \dashrightarrow \PP^1_{U}$ together with a $K$-isomorphism of the generic fiber $\mff_K$ with $f$. 
Such a model $\mff$ will be said to have \emph{good reduction} if the rational map $\mff$ can be extended to a $U$-morphism  $\mff: \PP^1_{U} \to \PP^1_{U}$; in this case, $\mff$ induces a morphism $\mff_v: \PP^1_{\kappa(v)} \to \PP^1_{\kappa(v)}$ fiber-wise for each $v \in U$. 

In particular, when $f$ has a model over $U = \Spec \mathfrak{o}_v$ with good reduction, then we say that \emph{$f$ has good reduction at~$v$}.  

A $K$-dynamical system $F=[ f ]_K$ is said to have \emph{good reduction} at a place $\mathnormal{v}$ of $K$ if $f $ has a model over $\ri{\mathnormal{v}}$ with good reduction.  

Sometimes this notion is called `simple good reduction' to distinguish it from other notions of good reduction that involve extra structure, such as marked critical or periodic points. 
\end{defn}

\begin{rem}\label{rem:goodIsTooWeak}
Whilst natural and appropriate for many situations, this notion of good reduction is not strong enough for a Shafarevich theorem: indeed, the dimension of the locus of monic polynomial maps of degree $d$ in the coarse moduli space $M_d$ is at least $d-3$, but all these have good reduction at all places. To obtain finiteness one must rigidify by adding extra structure and demanding that the additional structure also have good reduction in some sense; for example one might mark  the critical and branch points of the map; this was the approach of  Tucker and the first author in \cite{ST}.
\end{rem}

We introduce the notion of ``differential good reduction'', as a strengthening of the notion of ``critically good reduction'' in \cite{ST}.  Roughly speaking, a rational map has differential good reduction at a place $v$ if  no critical or branch points come together in the reduction mod $v$. First we fix some notation:

\begin{defn}\label{defn:critLocus}
Let $f$ be a rational map over $K$. Let $U$ be a scheme with field of functions $K$ and let $\mff$ be a model over $U$ with good reduction. Let $\mcR_{\mff}$ denote the ramification divisor of the $U$-cover $\mff$. Let $\mcB_{\mff}$ denote the branch divisor. We shall refer to $\mcC_{\mff}= \mcB_{\mff} + \mcR_{\mff}$ as the \emph{critical divisor} of $\mff$ and we shall refer to the associated reduced subscheme $C_\mff\subset \PP^1_K$ as the \emph{critical locus} of $\mff$.
\end{defn}

\begin{defn}\label{defn:DGR_char0}
Let $\mff: \PP^1_{U} \dashrightarrow \PP^1_{U}$ be a model for a rational map $f/K$. We say that $\mff$ has \emph{differential good reduction}, or \emph{D.G.R.} for short, if it satisfies the following three criteria: (1) $\mff$ has good reduction; (2) the maps $\mff_v$ are separable for each $v\in U$; (3) the critical locus $C_\mff$ is \'etale  over $U$. 

A $K$-dynamical system $F=[f]_K$ is said to have \emph{differential good reduction} at a place $\mathnormal{v}$ of $K$ if $f$ has a model with D.G.R. over $U = \Spec \ri{v}$.
\end{defn}

% If the base is the function field of a curve $X$ over a finite field, then we define D.G.R. only for models over open subsets $U$ of $X$, since we must also put tameness conditions at the places of $X$ not in $U$.   

% \begin{defn}\label{defn:DGR_charp}\lloyd{In char $p>0$ we want to say the the critical locus should be etale over U and also tamely ramified out side of U. But what if we want to define DGR at a single place?} 
% Suppose $K$ is the function field of a proper smooth curve $X$ over a finite field $k$. Let $U$ be a non-empty open subset of $X$. Let $\mff: \PP^1_{U} \dashrightarrow \PP^1_{U}$ be a model for a rational map $f/K$. We say that $\mff$ has \emph{differential good reduction} if it satisfies the following four criteria: (1) $\mff$ has good reduction; (2) the maps $\mff_v$ are separable for each $v\in U$; (3) the critical locus $C_\mff$ is \'etale  over $U$; (4) the Zariski closure of $C_\mff$ in $\PP^1_X$ is tamely ramified over $X$. \lloyd{we could also state this in terms of `ramification with bounded modulus'}

%A $K$-dynamical system $F=[f]_K$ is said to have \emph{differential good reduction} at a place $\mathnormal{v}$ of $K$ if $f$ has a model with D.G.R. over an open subset $U\subset X$ containing $v$.
% \end{defn}

\subsection{Admissible maps}

Let $K=k(X)$ be a function field as above, with $k$ algebraically closed. Then, by Lemma \ref{lem:spreadingOut} below, a rational map $f/K$ of degree $d$ has a model $\mathfrak{f}\: : \PP^1_U\to\PP^1_U$ with differential good (hence also simple good) reduction over some non-empty open $U\subset X$. Taking $v \mapsto [\mathfrak{f}_v]_k$ gives a morphism $\alpha: \: U\to M_d^\mathrm{dyn}$. This map is non-constant if and only if $f$ is non-isotrivial. 

\begin{defn}\label{def:Admissible} Let $f/K$ be a rational map, where $K$ is the function field of a curve over $k$ and $\Char K = p>0$. We shall say that $f/k$ is \emph{admissible} if it is non-isotrivial and moreover the map $\alpha$ associated to any model as above is separable.  If $K$ is a number field, then all rational maps $f/K$ maps shall be termed admissible.
\end{defn}

\begin{rem}
More generally one could define a ``modular inseparability index'' as in \cite[\S 0]{SzpiroSem},  and prove our results for maps of bounded modular inseparability index, but for simplicity we shall state the results in terms of admissible maps. 
\end{rem}

\subsection{Main result}

Denote by $\mcM^{\star}_d(K)$ the set of $K$-dynamical systems $F=[f]_K$ for which $f$ is a separable admissible rational map of degree $d$ with at least $3$ distinct ramification points.

The aim of this paper is to prove the following Shafarevich-style theorem for dynamical systems:

\begin{thm}\label{thm:Main}
Let $d\geq 2$ be an integer. Let $K$ be a number field or a function field over a field $k$. Suppose that $k$ is a field of characteristic zero or $p>2d-2$ that is either algebraically closed or finite.  Let $S$ be a non-empty finite set of places of $K$. Then the set 
\[
\Sh(d,K,S)=\left\{F \in \mcM^\star_d(K) \; : \;  F \text{  has D.G.R. outside of }S \right\}
\]
is finite.
\end{thm}

Roughly speaking, the theorem says there are only finitely many $K$-isomorphism classes of rational maps defined over $K$ that have differential good reduction outside of the given set of places.

Alongside Theorem \ref{thm:Main}, we shall also prove two other main results. The first result concerns the existence of global models with differential good reduction after finite bounded extension of the base field (Section \ref{s:GlobalModels}), the key ingredient for which is the ``smallness of the fundamental group'' for curves and rings of integers of number fields. The second is a finiteness  results for point sets and divisors on $\PP^1$ with prescribed bad reduction (Section \ref{s:Auxiliary}); here the key ingredient is the $S$-unit theorem and the moduli space $M_{0,n}$ for pointed $\PP^1$. 

 \begin{rem} 
 The strategy of proof for Theorem \ref{thm:Main} is as follows. For any rational map $f$ in $  \Sh(d,K,S)$, the fact that $f$ has differential good reduction outside $S$ implies that $C_f$ has good reduction (as a reduced effective divisor) outside of $S$. Applying our finiteness result for reduced effective divisors (Theorem  \ref{thm:finiteness}), we find a finite set $Y\subset \PP^1_K$ such that one can choose representatives $f$ for each equivalence class in $\Sh(d,K,S)$ with $C_f\subset Y$. To finish we apply  Mori's computation of the tangent space to the scheme of morphisms to deduce that the set of maps $f$ with $C_f$ contained in the finite set $Y$ is a finite set.  In the function field case, there are technicalities related to isotriviality and  characteristic $p$ phenomena.  

  \end{rem}

\subsection{Relation to other work on Shafarevich-type theorems} Shafarevich's original conjecture from \cite{Shafarevich} concerned finiteness of isomorphism classes of curves with a fixed finite set of primes of bad reduction over number fields and function fields. At the time he proved the case of elliptic curves over number fields. It is worth to note that his proof uses the $S$-unit theorem, and that the statement is false if ``elliptic curve'' is replaced with ``curve of genus one''. His conjecture was first proven for characteristic zero function fields by Parshin \cite{Parshin} and Arakelov \cite{Arakelov}; at this time it was noted by Parshin that Shafarevich's conjecture implied that of Mordell on the finiteness of rational points on curves of genus at least $2$. The case of function fields in characteristic $p$ was proven by the first author \cite{SzpiroSem}.  Faltings proved the conjecture for number fields  \cite{Faltings}. Recently, Yiwei She has proven Shafarevich results on K3 surfaces.

The  Shafarevich question for  rational maps was first addressed by Tucker and the first author in \cite{ST}. They introduced the notion of `critically good reduction' and proved a result for rational maps over number fields but with a different and weaker notion of isomorphism that the usual one used in dynamics; they consider rational maps up to the action of $\PGL_2\times\PGL_2$ with different action on the source and target $\PP^1$. The relation between critically good reduction and simple good reduction was studied in \cite{Canci}. Petsche \cite{Petsche} proved a Shafarevich result using the usual `dynamical' notion of isomorphism; his result works for certain families of maps over number fields and it uses a notion of `critically separable good reduction'. In the present work, which was presented at the birthday conference for J.H. Silverman in August 2015 and at the Fields Institute in February 2017 \cite{SzpiroToronto}, we make no restriction to families and we work with the dynamical notion of isomorphism; our result applies to number fields and function fields. 

Instead of marking critical points and critical values, one can mark some $n$-periodic points of the map; in \cite{PetscheStout} Petsche and Stout prove Shafarevich-style results about quadratic maps with marked 2-period or double fixed point structure. In the recent preprint \cite{SilvermanShaf},  Silverman, following similar lines to the present work, proves a joint generalization of the present work and that of \cite{PetscheStout} to obtain a very general Shafarevich finiteness result for marked rational maps on $\PP^1$, valid over number fields. 

\subsection{Relation to Shafarevich's theorem for elliptic curves}

Let $K$ be a number field and $S$ a finite set of finite places. It was noted in \cite{ST} that the finiteness of the set of isomorphism classes of elliptic curves over $K$ with good reduction outside of $S$ can be recovered from the Shafarevich theorem for rational maps applied to the Latt\`es map associated to multiplication by 2 on the elliptic curve. 

When $K$ is a function field over a finite field, the corresponding statement is true if one restricts to \emph{admissible} elliptic curves; i.e. those for which the modular map to the $j$-line is non-constant and separable -- see \cite{EllipticShaf} for a classical proof. 

If fact, one can derive the elliptic curve Shafarevich statement from our Theorem \ref{thm:Main}, as we now explain. Suppose $p\geq 5$ and $q=p^n$; let $K=\FF_q(X)$ be the function field of a curve and let $E/K$ be an elliptic curve.  After extending $S$ we can suppose that $\ri{s}$ is a PID and that $E$ has a minimal model over $\ri{s}$ of the form $y^2 = P(x)$ where $P(x)=x^3+Ax+B$. The admissibility of an elliptic curve $E$ corresponds to the condition that the Latt\`es map $f_E$ associated to multiplication by 2 on $E$ be admissible in the sense of Definition \ref{def:Admissible}. The map $f_E$ is ramified at the six $x$-coordinates of the twelve points $E[4]\setminus E[2]$, whilst the branch points are the $x$-coordinates of the three points $E[2]\setminus\{\mathcal{O}\}$. Therefore the critical set $C_{f_E}$ is the union of the roots of the polynomials $\phi_2=x^3+Ax+B$ and  $\phi_4=\psi_4/\psi_2$, where $\psi_n$ denotes the $n$-th division polynomial. The discriminant $\disc(\phi_2\phi_4)$ equals $\Delta_E^{12}$ up to a power of 2, where $\Delta_E$ denotes the usual discriminant of $E$. Therefore differential good reduction of $f_E$ is equivalent to good reduction of $E$, at least away from characteristic 2 and 3.  By Theorem \ref{thm:Main}, there are at most finitely many isomorphism classes of admissible Latt\`es maps corresponding to admissible elliptic curves with good reduction outside of $S$.  On the other hand,  if $f$ is a Latt\`es map with denominator $P(x)$, we recover an elliptic curve with equations $y^2=dP(x)$ defined up to $d\in \ri{s}^\times/(\ri{S}^\times)^2$. Since   $\ri{s}^\times/(\ri{S}^\times)^2$ is finite, we have at most finitely many  isomorphism classes of  admissible elliptic curves  with good reduction outside of $S$.

\section*{Acknowledgments}

\thanks{The authors would like to thank Joseph Gunther for helpful comments. We also thank Thomas Tucker for his contribution to the previous paper \cite{ST}, which led to this one.

The first author is partially supported by NSF Grants DMS-0739346 and PSC CUNY grant 680310046.
%%%%%%%%

\section{Global Models}\label{s:GlobalModels}

Throughout this section, the notation $X, K,$ etc. are as in Section \ref{ss:ModelsReduction}.

\begin{lem}\label{lem:spreadingOut}
If a $K$-dynamical system $F=[f]_K$ has good reduction (resp. D.G.R.) at a place $v$, then there exists an open set $V\subset X$ containing $v$ such that $f$ has a model with good reduction (resp. D.G.R.) over $V$. 
\end{lem}
\begin{proof} 
By Definition \ref{defn:DGR_char0}, there is a model for $f$ over $\Spec \ri{v}$; i.e. there exists a pair of homogeneous polynomials, $f_1,f_2 \in \mathfrak{o}_v[x,y]$, that define a morphism $ \mathfrak{f}: \PP_{\mathfrak{o}_v} \to \PP_{\mathfrak{o}_v}$ whose generic fiber is $K$-isomorphic to $f$. This automatically extends to a rational map $\mff: \PP^1_{U} \dashrightarrow \PP^1_{U}$ for any open $U\subset X$ containing $v$. The resultant $\Res(f_1,f_2)$ is an element of $\mathfrak{o}_U$. For a place $v^\prime$ in $U$, the non-vanishing of $\Res(f_1,f_2)$ implies that $\mathfrak{f}$ extends to a morphism over $v^\prime$ (\cite[Section 2.5]{SilvermanBook}); since this is an open condition, we obtain the desired open set for the good reduction part of the statement. The condition of having D.G.R. is also easily seen to be an open condition: see Section \ref{s:Equations} for equations defining differential bad reduction in special cases. 
 \end{proof}

\begin{rem}\label{r:bundle}
Suppose $f$ has D.G.R outside of the finite set $S$. Then by lemma \ref{lem:spreadingOut}, and since open sets of $X^\circ$ are co-finite in $X^\circ$, we can find an open cover $X^\circ= \cup_i U_i$ and models $\mff_i$ for $f$ over $U_i$ with D.G.R. For each such $\mff_i$, we have an isomorphism $\gamma_i: (\PP^1_{U_i})_K \xrightarrow{\sim} \PP^1_K$. From this data we can construct a vector bundle $\mcE$ of rank 2 and an $X^\circ$-morphism $\mff: \PP(\mcE)\to \PP(\mcE)$ that restricts to $\mff_i$ above each $U_i\subset X^\circ$.    
\end{rem}

We now prove that, in certain cases, one can trivialize the bundle and obtain a global model $\mff: \PP^1_{X^\circ} \to \PP^1_{X^\circ}$ with D.G.R. 

\begin{prop}\label{prop:GlobalModels}
Let $K$ be the function field of a Dedekind scheme $X$. Let $S$ be a non-empty finite set of places of $K$. Let $f\in \Sh(d,K,S)$ be a $K$-dynamical system with D.G.R. outside of $S$. Then there exists an \'etale cover $C\to X^\circ$ of degree at most $(2d - 2)!$ and a global model $\mff: \PP^1_{C} \to \PP^1_{C}$ with D.G.R. that is a model for the morphism $f/K^\prime$, where $K^\prime = k(C)$ is the function field of $C$. 
\end{prop}

\begin{proof}
We have a vector bundle $\mcE$ of rank 2 and a $X^\circ$-morphism $\mff: \PP(\mcE)\to \PP(\mcE)$, as in Remark \ref{r:bundle} above. Our strategy is to trivialize the $\PP^1$-bundle $\PP(\mcE) \to X^\circ$ by an  \'etale base change.  

Differential good reduction implies that the critical locus $C_\mff$ is \'etale over $X^\circ$. If  $C_\mff$ is split \'etale over $X^\circ$, then we are done: the sheets of the cover $C_\mff \to X^\circ$ are given by disjoint sections $\sigma_i: X^\circ \to \PP(\mcE)$; there are at least three such sections, by the assumption that $f$ has at least three distinct ramification points; since an automorphism of $\PP^1$ is determined by the image of three points, sending any three of these sections to $0,1,\infty$ in each fiber gives a trivialization of the $\PP^1$-bundle $\PP(\mcE) \to X^\circ$.      

If some component $C$ of the critical locus is not split \'etale over $X^\circ$, we can make the base change by $C \to X^\circ$ (whose degree is at most $(2d-2)!$), so that the corresponding component of the critical locus in $\pi:\,\PP(\mcE)\times_{X^\circ} C \to C$ becomes split \'etale over $C$. Now we get a trivialization as before.   
\end{proof}

\begin{rem} A scheme that has only finitely many \'etale coverings of given degree is said to have ``small \'etale fundamental group''.  Examples include: (1) the spectrum of the ring of integers of a number field (Hasse-Minkowski); (2) a curve over an algebraically closed field of characteristic zero (Grothendieck). In general, the \'etale fundamental group of an affine curve in characteristic $p$ is not small, due to wild ramification. Nonetheless, if $U\subset X$ is an open subset of a smooth proper curve over a finite field, there are at most finitely many tamely ramified covers of $X$ that are \'etale over $U$.  More generally, one can prove finiteness of the set of covers of $U$ with ``ramification bounded by a modulus''; see \cite{smallPi1}. 
\end{rem}

\begin{cor}\label{cor:UniformGlobalModels}
Let $X$ be the spectrum of the ring of integers of a number field or a smooth proper curve over a field $k$ that is either algebraically closed of characteristic zero, or a finite field of characteristic greater than $2d-2$. Let $K$ be the function field of $X$. Let $S$ be a finite set of places of $K$. Then there exists a finite extension $L/K$ such that any element $f\in \Sh(d,K,S)$ has a global model over $L$ with D.G.R. and such that the ramification points of $f$ are rational over $L$.     
\end{cor}
\begin{proof}
Apply the preceding proposition. Note that, by the hypothesis on the characteristic, any cover $C\to X$ of degree at most $(2d-2)!$ is tamely ramified. Take $L$ to be the compositum (in some fixed separable closure of $K$) of the extensions $K(C)/K$ for the finitely many tamely ramified covers $C\to X$ of degree at most $(2d-2)!$ that are \'etale over $X^\circ$. 
\end{proof}

\begin{prop} 
Let $K$ be the function field of a curve $X/k$ such that the Jacobian is 2-divisible (i.e. for all $P\in J_X(k)$ there exists $Q\in J_X(k)$ with $2Q=P$). Let $S$ be a nonempty finite set of points of $X$ that contains at least one $k$-rational point. Then any $\PP^1$-bundle over $X^\circ = X\setminus S$ is trivial.    
\end{prop}
\begin{proof}
Note that $X^\circ= \Spec \ri{S}$ is affine and write $A= \ri{S}$.  A $\PP^1$-bundle over $X^\circ$ is determined by a vector bundle $E$ of rank two on $X^\circ$; i.e. the $\PP^1$-bundle is of the form $\PP(E)\to X^\circ$. For any line bundle $M$ on $X^\circ$, the $\PP^1$-bundle $ \mathbb{P}(E) $ is canonically isomorphic to $ \mathbb{P}(E \otimes M) $. Therefore our strategy will be to replace $E$ by $E \otimes M$ for some line bundle $M$ so that $E \otimes M$ becomes trivial. 

Since $X^\circ$ is affine and one dimensional and $E$ is rank two, $E$ will have a nowhere vanishing section $A  \rightarrow E$ (\cite[Th\'eor\`eme 1]{Serre}). So we have an exact sequence 
$$
0 \rightarrow A \rightarrow E \rightarrow \Lambda^2(E) \rightarrow 0
$$
If we denote the line bundle $\Lambda^2(E)$ by $L$, then we have, for any line bundle $M$
$$
\Lambda^2(E\otimes M)\isom L\otimes M^{\otimes 2}
$$
We claim that $L$ is a square in the Picard group; then we can take $M= L^{\otimes \frac{1}{2}}$ to complete the proof (since then the exact sequence above will split). To see that $L$ is indeed a square, first extend it to a line bundle $\overline{L}$ over $X$. Suppose the degree of $\overline{L}$ is $n$, then take $L^\prime = \overline{L}(-n P)$, where $P$ is a $k$-rational point in $S$. Now $\deg L^\prime =0$ and corresponds to an element of $J_X(k)$, so by the assumption on the divisibility of the Jacobian, $L^\prime$ has a square root $N$. Now $N|_{X^\circ}$ is the desired square root of $L$.
\end{proof}

\begin{cor}
Let $X$ be a curve over an algebraically closed field $k$. Let $K$ be the function field of $X$ and let $S$ be a non-empty finite set of places. Then any element $f\in \Sh(d,K,S)$ has a global model over $K$ with D.G.R.    
\end{cor}
\begin{proof}
For $k$ algebraically closed, $J_X(k)$ is a divisible group; so we can apply the preceding proposition to find a trivialization of the bundle that occurs in Remark \ref{r:bundle}, thereby obtaining the desired global model.
\end{proof}

%%%%%%%%%%%%%%%%%%
\section{Differential Discriminant}\label{s:Equations}

Let $F/K$ be a rational map of degree $d$ written in coordinates as $F(x_0,x_1)= [F_0(x_0,x_1), F_1(x_0,x_1)]$ where $F_0$ and $F_1$ are homogeneous polynomials of degree $d$. The ramification locus $R_F$ of $ F$ in $\PP^1$ is given by the vanishing of the \emph{Wronskian}
 
\[
w_F(x_0,x_1)= \det\left(\frac{\partial F_i}{\partial x_j}\right)
\]

Let $g(x_0,x_1,y_0,y_1) = y_1F_0(x_0,x_1)-y_0F_1(x_0,x_1)$. This is the form defining the graph of $ F$. The branch locus $B_ F$ is the image of the ramification locus, hence it is given by the vanishing of the resultant in $(x_0,x_1)$ of $w_ F(x_0,x_1)$ and $g(x_0,x_1,y_0,y_1) $; that is, by the vanishing of
\[
b_ F(y_0,y_1)=\Res_x(y_1F_0(x_0,x_1)-y_0F_1(x_0,x_1), w_ F(x_0,x_1))
\]

The critical locus $C_ F= B_ F \cup R_ F$ is given by the vanishing of $u_ F(x_0,x_1)w_ F(x_0,x_1)$. Define the \emph{differential discriminant}
\[
\Delta_{\mathrm{diff}}(F) = \disc_x\left(b_F(x_0,x_1)w_F(x_0,x_1)\right)
\]

By definition, the model $ F$ has differential good reduction at $v$ when $|\red_v(C_ F)| = |C_ F|$. If $ F$ has the maximum number of distinct ramification and branch points -- i.e. if $|C_ F|=4d-4$ -- then we call $F$ \emph{differentially separated} (cf. the notion of critically separated in \cite{Petsche}). In that case we have
\[
 F\text{ has D.G.R. at } v \quad \Longleftrightarrow \quad \ord_v( \Delta_{\mathrm{diff}}( F)) = 0
\]

Denote the coefficients of the forms $F_0$ and $F_1$ by $a_1,\dots a_{2d+1}$ and let $\mathrm{Form}_d$ be the affine space $\Spec \ZZ[a_1,\dots a_{2d+1}]$; that is, $\mathrm{Form}_d$ is the space of pairs of homogeneous forms of degree $d$. There is an action of $\mathbb{G}_m\times \GL_2$ on $\mathrm{Form}_d$ given by 
\[
 F^{(\alpha, \Gamma)} =[F_0,F_1]^{(\alpha, \Gamma)} = [\alpha F^\Gamma_0,\alpha F^\Gamma_1]
\]
for $(\alpha, \Gamma) \in \mathbb{G}_m\times \GL_2$, where  $[F^\Gamma_0, F^\Gamma_1]=\Gamma\circ [F_0,F_1]\circ \Gamma^{adj}$.

Under this action, $\Delta_{\mathrm{diff}}( F) $ is a relative invariant; that is, 
\[
\Delta_{\mathrm{diff}}( F^{(\alpha, \Gamma)}) = \alpha^n \det(\Gamma)^m \Delta_{\mathrm{diff}}( F) 
\]
for positive integers $n$ and $m$.

\begin{comment}
Let $S=\Spec A$, where A is a Dedekind ring with fraction field $K$ and take  $f\in \mathrm{Rat}_d(K)$. 

For a critically separated map $f$, define a divisor $\Delta_{\mathrm{min}}(f)$ on $S$ by 
\[
\ord_v(\Delta_{\mathrm{min}}(f)) = \min\{\ord_v(\Delta_{\mathrm{diff}}(F) )\,:\,  F=[F_0,F_1] \text{ is a model with good reduction for } f \text{ over } S\}
\]
Say that $F$ is a \emph{minimal model} for $f$ if 
\[
\ord_v(\Delta_{\mathrm{min}}(f)) =\ord_v(\Delta_{\mathrm{diff}}(F) )
\] 
for all $v$.

\begin{prop}
If $S = \Spec A$, where $A$ is a principal ideal domain with field of fractions $K$, then any critically separated rational map of degree $d$ over $K$ has an $S$-minimal model for CER.
\end{prop}
\begin{proof}
I think this goes through  as in \cite{BruinMolnar} just by the formal properties of $\Delta_{\mathrm{diff}}( F) $.\lloyd{check or remove}
\end{proof}

\begin{prop}
For a critically separated rational map, critically excellent reduction at $v$ implies good reduction at $v$
 \end{prop}
\begin{proof}
Let $F = [F_0(x_0,x_1), F_1(x_0,x_1)]$ be a model for $\phi$ realizing critically excellent reduction. Then we have $ord_v(\Delta_{\mathrm{diff}}(F))=0$. On the other hand [lemma] $ \Res(F_0,F_1)$ divides $\Delta_{\mathrm{diff}}(F)$, so we must also have $ord_v(\Res(F_0,F_1))=0$; i.e. good reduction at $v$.
\end{proof}

\end{comment}

\begin{ex}(Quadratic maps)
For a map of degree 2 one can calculate 
\[
\Delta_{\mathrm{diff}}( F) = 64 \rho^8\theta_1^2 \theta_2^2
\]
where  $\rho$ is the resultant and $\theta_i$ are relative invariants given by 
\begin{eqnarray*}
\theta_1 &=& \tau_1 -2 \rho\\
\theta_2 &=& \tau_1 +6 \rho
\end{eqnarray*}
where $\tau_1$ is the numerator of the rational function of the coefficients given by the first symmetric function in the multipliers of the map (see \cite{SilvermanSpace} and \cite{West} for more on the invariants of rational maps). 

\end{ex}

\begin{ex}(Quadratic map with distinct critical points)
After moving the two critical points to $[0:1]$ and $[1:0]$, one can put it into the form
\[
F = [ax_0^2+cx_1^2\,:\, dx_0^2+fx_1^2]
\]
For such a map, one has 
\[
\Res(F) = (ad-cf)^2
\]
and
\[
\Delta_{\mathrm{diff}}(F) = 64a^2c^2d^2f^2(ad-cf)^{20}
\]
\end{ex}

%%%%%%%%%%%%%%%%%%
\section{Auxiliary finiteness result }\label{s:Auxiliary}

First recall the moduli space of $n$ points on the projective line. 
Consider the projective line over a scheme $Y$: i.e. $\PP^1_Y \to Y$. Let $\mcP_{0,n}$ be the functor that assigns to any scheme $Y$ the set of all $n$-tuples $\boldsymbol{\sigma}=(\sigma_i:\, Y\to \PP_Y^1)_{i=1}^n$ of disjoint sections of the structure morphism $\PP_Y^1 \to Y$. This is represented by the scheme 
\[
P_{0,n}=(\PP^1)^n  \setminus \text{diagonals}
\] 
\begin{defn}
Define two $n$-tuples of disjoint sections, $\boldsymbol{\sigma}=(\sigma_i: Y\to \PP_Y^1)_{i=1}^n$ and $\boldsymbol{\sigma^\prime}=(\sigma^\prime_i: Y\to \PP_Y^1)_{i=1}^n$, to be \emph{$Y$-equivalent} if there exists an automorphism ${\gamma \in \Aut_Y(\PP^1_Y) = \PGL_2(Y)}$ which moves one $n$-tuple into the other; i.e. ${\sigma_i =\gamma\circ \sigma_i^\prime}$. If $Y = \Spec A$ is affine, $Y$-equivalent $n$-tuples shall also be said to be $A$-equivalent. 
\end{defn}

Let $\mcM_{0,n}$ denote the functor that, to any scheme $Y$, assigns the set of equivalence classes of elements of $\mcP_{0,n}(Y)$. For any $n\geq 4$,  $\mcM_{0,n}$ is represented by the scheme 
\[
M_{0,n}= \left(M_{0,4}\right)^{n-3} \setminus \text{diagonals}
\]
where
\[
M_{0,4} = \PP^1 \setminus \{0,1,\infty\}
\]

If we write a point of $P_{0,n} $ as $\mathbf{s} = (s_i)_{i=1}^n$, where $s_i=[\alpha_i: \beta_i] \in \PP^1$, then the map 
\[
\pi: P_{0,n} \to M_{0,n}
\]
taking an $n$-tuple to the corresponding point in the moduli space is given by 
\[
\mathbf{s} \mapsto \left([123i](\mathbf{s})\right)_{i=4}^n
\]
The notation $[ijkl](t)$ denotes the cross-ratio
\[
[ijkl](\mathbf{s}) = \frac{[ik][jl]}{[il][jk]}(\mathbf{s})
\]
where
\[
[ij](\mathbf{s})= \alpha_i\beta_j - \alpha_j\beta_i
\]
Note that we can treat $[ijkl](t)$ as a point of $\PP^1\setminus \{0,1,\infty\}$. \\

In the sequel we work over a scheme $X$ that we take to be the spectrum of the ring of integers of a number field  or a  smooth proper curve over a  field $k$, where $k$ is either a finite field or an algebraically closed field. Let $K$ be the function field of $X$. Let $S$ be a non-empty finite set of (finite) places of~$K$.

\begin{defn}\label{def:AdmissibleTuple}
Let $X$ be a curve  defined over a field $k$, we say that an $n$-tuple ${\boldsymbol{\sigma}=(\sigma_i: \Spec \ri{S}\to \PP_{ \ri{S}}^1)_{i=1}^m}$ is  \emph{$\ri{S}$-trivial} if it is $\ri{S}$-equivalent to an $n$-tuple consisting of constant sections. Note that in this case, the restriction of the sections to the generic fiber can be specified as a set of $k$-points of~$\PP^1$.  

A non-$\ri{S}$-trivial $n$-tuple ${\boldsymbol{\sigma}=(\sigma_i: \Spec \ri{S}\to \PP_{ \ri{S}}^1)_{i=1}^m}$ induces a non-constant morphism of $X^\circ$  onto a curve in the moduli space $M_{0,n}$. If, moreover, this morphism is separable, then we shall say that ${\boldsymbol{\sigma}}$ is \emph{admissible}.
\end{defn}

\begin{prop}\label{prop:finiteSections}
Let  $X,K,S$ be as above and let $n$ be a positive integer.  Then there are  at most finitely many $\ri{S}$-equivalence classes of $n$-tuples $${\boldsymbol{\sigma}=(\sigma_i: \Spec \ri{S}\to \PP_{ \ri{S}}^1)_{i=1}^m}$$ of disjoint admissible sections.  
\end{prop}

\begin{proof}
Denote by $\mathcal{F}_n$ the set of all such $n$-tuples $\boldsymbol{\sigma}$ satisfying the hypotheses of the proposition. We must show that the image of $\mcF_n$ under $\pi$ is a finite set.

By the explicit description of $\pi$ in terms of cross-ratios, it suffices to show that the set of cross-ratios 
\[
\mcG_n =  \left\{ [ijkl](\boldsymbol{\sigma}) \; :\; \boldsymbol{\sigma}\in \mcF_n\right\}   
\]
is finite. The elements of $\mcG_n$ are $\ri{T}$-points of $\PP^1\setminus \{0,1,\infty\}$; hence they are $T$-units. Moreover, the cross-ratios satisfy the identity 
\[
[ijkl](\boldsymbol{\sigma})+[ikjl](\boldsymbol{\sigma})=1
\]
In the case that $K$ is a number field, we can now conclude finiteness of $\mcG_n$ from the finiteness of solutions to the unit equation (see \cite{EGbook} for an exposition of the unit equation, generalizations and effective versions). In the case of a function field $K$, we cannot immediately conclude finiteness: if we set 
$A=k^\times$ if $\Char K = 0$ and $A=(K^\times)^p$ if $\Char K = p>0$, then the unit equation may have infinitely many solutions in $A$.  

To deal with that issue, consider the set 
\[
\mcH_n = \left\{ [ijkl](\boldsymbol{\sigma}) \; :\; \boldsymbol{\sigma}\in \mcF_n, \text{ and }    [ijkl](\boldsymbol{\sigma})\notin A \right\}   
\]

The set of solutions to the unit equation $u+v=1$  with $u,v\in \ri{T}^\times\setminus A$ is finite (see \cite[Theorem 7.6.1]{EGbook}), therefore $\mcH_n$ is a finite set.

This implies that the set $\mcG_n$ is finite, since
\[
\mcG_n \subset \left\{xy\; : \; x,y \in \mcH_n\cup\{1\} \right\}
\]
To see this, note the identity of cross ratios
\[
[123i](\mathbf{t})=[123l](\mathbf{t})\cdot[12l i](\mathbf{t})
\]
Secondly, observe that, since $M_{0,n}$ is a fine moduli space and $\boldsymbol{\sigma}$ is admissible, for any $\boldsymbol{\sigma} \in \mcF_n$, there exists an index $\ell$ such that $[123\ell](\boldsymbol{\sigma})\notin A$. Then, for any index $[123i](\boldsymbol{\sigma}) \in \mcG_n$, either $[123i](\boldsymbol{\sigma})\notin A$, so we are done, or  $[123i](\boldsymbol{\sigma})\in A$. In the latter case, the equation 
\[
[123i](\mathbf{t})=[123\ell](\mathbf{t})\cdot[12\ell i](\mathbf{t})
\]
together with the fact that $[123\ell](\boldsymbol{\sigma})\notin A$ implies that $[12\ell i](\boldsymbol{\sigma})\notin A$, so $[123i](\boldsymbol{\sigma})$ is indeed of the form $xy$ with $ x,y \in \mcH_n$. 
\end{proof}

\begin{rem}
This proof is a generalization of the one given in \cite{EG} for finiteness of polynomials with given discriminant.
\end{rem}

%Suppose $X,S,K$ be as above. Let $L/K$ be a finite extension. Write $T$ for the set of places of $L$ above $S$. We shall say that two $n$-tuples $\boldsymbol{\sigma},\,\boldsymbol{\sigma}^\prime\in \mcP_{0,n}(\ri{T})$ are $\ri{S}$-equivalent if there is an element $\gamma \in \Aut(\PP^1_{\ri{S}})$ that (considered as an automorphism of $\PP^1_{\ri{T}}$ under base change) takes  $\boldsymbol{\sigma}$ to $\boldsymbol{\sigma}^\prime$. 

%Let ${\mcF\subset \mcP_{0,n}(\ri{S})}$ be a subset of $n$-tuples with the property that the elements of tuple are permuted under the action of $\Gal(L/K)$. Then

\begin{defn}
Let $D$ be a reduced effective divisor on $\PP^1_K$; i.e. a \mbox{$\Gal(\overline{K}/K)$-stable} set of points in $\PP^1(\overline{K})$. For an extension $L/K$, we say two such reduced effective divisors $D$ and $D^\prime$ on $\PP^1_K$ are \emph{$L$-equivalent} if there exists an element $\gamma \in \PGL_2(L)$ that takes the points of $D$ to those of $D^\prime$. 
\end{defn}

\begin{lem}\label{lem:finiteExtension}
Let $L/K$ be a finite extension. Let $\mcF$ be a set of reduced effective divisors of degree $n\geq 3$ on $\PP_K^1$ such that the points of each $D\in \mcF$ are defined over $L$. Then each $L$-equivalence class in $\mcF$ consists of at most finitely many $K$-equivalence classes.
\end{lem}
\begin{proof}
(C.f. the proof of \cite[Theorem 4]{Petsche}.) Let  $c$ be an $L$-equivalence class in $\mcF$ and  fix some $D_0\in c$; note that $D_0$ is a set $\bar{D}_0$ of $n$ points  in $\PP^1(\overline{K})$. For any other element $D\in c$, corresponding to some set of points $\bar{D}\subset \PP^1(\overline{K})$, there exists a $\gamma_{D} \in \PGL_2(L)$ that maps $\bar{D}$ bijectively to $\bar{D}_0$. Note that $\Gal(L/K)$ acts on the set $\bar{D}$. Write $\Gal(L/K)=\{\tau_1,\dots,\tau_m\}$,  and define a bijection $\iota_D : (\bar{D}_0)^m\to (\bar{D}_0)^m$ by
$$
\iota_D(x_1,\dots, x_m) = (\gamma_{D}\circ\tau_1\circ\gamma_{D}^{-1}(x_1), \dots , \gamma_{D}\circ\tau_m\circ\gamma_{D}^{-1}(x_m))
$$
We claim that $\iota_{D} = \iota_{E}$, implies that $D$ is $K$-equivalent to $E$, for $D$ and $E$ in $\mcF$. Since there are only finitely many possible bijections from the finite set  $\bar{D}_0$ to itself, this claim will complete the proof. 

To prove the claim, suppose $\iota_{D} = \iota_{E}$ and let $\gamma = \gamma_D\circ \gamma_{E}^{-1}$. Then $\gamma$ is an $L$-equivalence between $D$ and $E$. We must show that $\gamma \in \PGL_2(K)$; i.e. we want to show that $\gamma = \prescript{\tau}{}{\gamma}$ for all $\tau\in \Gal(L/K)$. Since $\iota_{D} = \iota_{E}$, we have $\gamma \circ \tau (x) = \tau\circ\gamma (x)$ for each $x\in \bar{E}$ and for each $\tau\in \Gal(L/K)$; so $\gamma $ and $\prescript{\tau}{}{\gamma}$ take the same values on the set $\tau(\bar{E})$. We can conclude that $\gamma = \prescript{\tau}{}{\gamma}$, since $\tau(\bar{E})$ contains at least three points and an element of $\PGL_2(L)$ is uniquely determined by where it sends three points. 
\end{proof}

\begin{defn}
 We shall say that a reduced effective divisor $D$ on $\PP^1_K$ is \emph{$K$-trivial} if there exists a reduced effective divisor $D_0$ on $\PP^1_k$, i.e. defined over the base field $k$, such that $D$ is $K$-equivalent to $(D_0)_K$. We shall say that $D$ is \emph{isotrivial} if there exists a reduced effective divisor $D_0$ on $\PP^1_{\overline{k}}$, such that $D_{\overline{K}}$ is $\overline{K}$-equivalent to $(D_0)_{\overline{K}}$. 
Define the \emph{splitting degree over $K$} of a reduced effective divisor $D$ on $\PP^1_K$ to be the minimum degree of an extension $L/K$ such that each of the points of $D$ is $L$-rational; note that this is at most $(\deg D)!$. 
\end{defn}

\begin{defn} 
Let $D$ be a reduced effective divisor on $\PP^1_K$ and let $U$ be an open subset of $X$. By a \emph{model $(\mathfrak{D}, \varphi)$ for $D$ over $U$} we shall mean a divisor $\mathfrak{D}$ of $\PP^1_U$ that dominates $U$, together with a $K$-isomorphism $\varphi: \PP^1_K \xrightarrow{\sim} (\PP^1_U)_K$ under which $\mathfrak{D}$ pulls back to $D$.  We shall say that $D$ has \emph{good reduction} at a place $v$ of $K$ if there is a model $\mathfrak{D}$ for $D$ over an open set $U$ containing $v$ for which $\mathfrak{D}$ is \'etale over $U$. Intuitively, this means that the ``reduction modulo $v$'' of $\Supp D$  has the same number of points as $\Supp D$.
\end{defn}

Let $n\geq 3$ be an integer. Denote by $\Sh_\mathrm{div}(n,\delta, K,S)$ the set of non-isotrivial reduced effective divisors of $\PP^1_K$ that have good reduction at each place $v$ not in $S$, that are supported on at least three points, that have degree at most $n$ and splitting degree over $K$ at most $\delta$.

\begin{prop}\label{prop:modelsForDivisors}
Let $n,\delta,K,S$ be as above. Suppose $\Char K > \delta$. Then there exists a finite extension $L/K$ such that any element of $\Sh_\mathrm{div}(n,\delta, K,S)$ has a model over $X_L^\circ = X_L \setminus T$ that is split \'etale over $X^\circ_L$, where $T$ is the set of places of $L$ above $S$.  
\end{prop}
\begin{proof}
Under the hypotheses of the proposition, $X^\circ$ has small fundamental group, so the statement can be proved along the same lines as the proofs of Proposition \ref{prop:GlobalModels} and Corollary \ref{cor:UniformGlobalModels}.
\end{proof}

\begin{rem}
Let $L/K$ be the finite extension whose existence is asserted in Proposition \ref{prop:modelsForDivisors}.  For each element $D$ of $\Sh_\mathrm{div}(n,\delta, K,S)$, we can take a split \'etale model $\mathfrak{D}$ for $D$ over $X^\circ_L = \Spec \ri{T}$ as in Proposition \ref{prop:modelsForDivisors}. Such a $\mathfrak{D}$ is the image of a set  $${\boldsymbol{\sigma}=\{\sigma_i: \Spec \ri{T}\to \PP_{ \ri{T}}^1\}_{i=1}^m}$$ 
 of $m$ disjoint non-$\ri{T}$-trivial sections, where $ 3\leq m\leq n$; after arbitrarily assigning a ordering, this gives us an element of $\mcP_{0,m}(\ri{T})$. 
\begin{defn}
We shall say that  an element $D$ of $\Sh_\mathrm{div}(n,\delta, K,S)$ is \emph{admissible} if for some choice $\boldsymbol{\sigma}$ of $m$-tuple associated to $D$ as above, $\boldsymbol{\sigma}$ is an admissible $m$-tuple in the sense of Definition \ref{def:AdmissibleTuple}. Denote that set of admissible elements of $\Sh_\mathrm{div}(n,\delta, K,S)$ by $\Sh^\star_\mathrm{div}(n,\delta, K,S)$.
\end{defn}

 For each $m$ with $ 3\leq m\leq n$, denote by $\mcF_m$ the set of admissible $m$-tuples of sections obtained for each element of $\Sh^\star_\mathrm{div}(n,\delta, K,S)$ as just described; thus we have a bijection $\theta: \Sh^\star_\mathrm{div}(n,\delta, K,S) \to \coprod\mcF_m$.  By  Proposition \ref{prop:finiteSections}, there are  at most finitely many $\ri{T}$-equivalence classes in $\mcF_m$. Two divisors $D$ and $D^\prime$ are $L$-equivalent if $\theta(D)$ and $\theta(D^\prime)$ are $\ri{T}$-equivalent. Therefore we have shown that there are at most finitely many \mbox{$L$-equivalence} classes in $\Sh^\star_\mathrm{div}(n,\delta, K,S)$.  By Lemma \ref{lem:finiteExtension}: each $L$-equivalence class splits into at most finitely many \mbox{$K$-equivalence} classes, so we have proved the following:
\end{rem}

\begin{thm}\label{thm:finiteness}
Let $n\geq 3$, $\delta,\,K$ and $S$ be as above. Suppose $\Char K > \delta$ or $\Char K = 0$. There are at most finitely many $K$-equivalence classes of elements of $\Sh^\star_\mathrm{div}(n,\delta, K,S)$. 
\end{thm}

\section{Rigidity}

\begin{prop}\label{prop:rigidity}
Fix a finite set $Y \subset \PP^1(\bar{K})$. There exist only finitely many morphisms $f: \PP_K^1 \to \PP_K^1$ satisfying
\begin{enumerate}
\item $R_f \subset Y$
\item $B_f \subset Y$
\item $| R_f |\geq 3$
\end{enumerate}
\end{prop}
\begin{proof}
This is proved as Proposition 4.2 in \cite{ST}.  We just note that the proof works by application of the following theorem of Mori, which computes the tangent space to the scheme of morphisms between two varieties  \cite{Mori}: 

\begin{thm}
Let $A$ and $B$ be schemes of finite type over a  field $k$ and let $Z$ be a closed subscheme of $A$. Let $p: Z \to B$ be a $k$-morphism. Denote by $\Hom_k(A,B;p)$ the set of $k$-morphisms from $f: A\to B$ that extend $p$; i.e. $f|_Z=p$. Denote the ideal sheaf of $Z$ in $A$ by $\mathfrak{I}_Z$. For a scheme $X/k$ denote  the tangent space of $X$ over  $k$ by $T_X$. Then $\Hom_k(A,B;p)$ is represented by a closed subscheme of $\Hom_k(A,B)$ and for any closed point $f$ of $\Hom_k(A,B;p)$, we have 
\[
T_{\Hom_k(A,B;p),f}\isom H^0(A, f^*T_B\otimes_{\mathcal{O}_A} \mathfrak{I}_Z)
\]
\end{thm}

We take $A$ and $B$ to be $\PP^1$ and $Z$ is the ramification divisor. In this case, the dimension of the tangent space turns out to be zero, so there are at most finitely many morphisms with the specified ramification. 
\end{proof}

\section{Proof of Main Theorem}

In this section we keep the hypotheses of Theorem \ref{thm:Main}.

\begin{lem}\label{lem:nonisotrivial}
Suppose $f/K$ is a rational map with at least three distinct ramification points whose critical locus $C_f$ is not admissible (when considered as a reduced effective divisor on $\PP^1_K$). Then $f$ is not admissible.  
\end{lem}
\begin{proof}
Assume for simplicity that $f$ is differentially separated and of degree $d$. Let $M_d^{\mathrm{ds}}$ denote the dense open subset of the moduli space $M_d^{\mathrm{ds}}$ or rational maps of degree $d$ that are differentially separated. Let $n = 4d-4$. Let $M_{n,0}^{\mathrm{sym}}$ denote the moduli space of $n$ unordered points on $\PP^1$. Choosing a model with differential good reduction of $f$ over an open set $U$ of $X$, we get modular maps $\alpha:\: U\to  M_d^{\mathrm{ds}}$ and $\beta :\: U\to M_{n,0}^{\mathrm{sym}}$. Admissibility of $f$ amounts to $\alpha$ being a non-constant separable map onto its image; admissibility of $C_f$ amounts to $\beta$ being a non-constant separable map onto its image. By Proposition \ref{prop:rigidity}, the map $\pi:\;M_d^{\mathrm{ds}} \to M_{n,0}^{\mathrm{sym}}$, that takes the class $C_f$ of a rational map to the class of $\mathrm{Supp}(C_f)$, is quasi-finite onto its image; moreover its restriction to the image of $\alpha$ is separable, since by hypothesis, the characteristic is zero or greater than $4d-4$. Since $\pi, \alpha$ and $\beta$ form a commuting triangle, the condition that $\beta$ be non-constant and separable onto its image is equivalent to the same condition on $\alpha$. 

If $f$ is not differentially separated, say $|C_f|=m<4d-4$, restrict to the stratum of $M_d$ consisting of such maps and apply the same argument as that just given, with $m$ in place of $n$. 

%If $C_f$ is  isotrivial, then there exists $\gamma \in \PGL_2(\overline{K})$ such that $\gamma(C_f)$ is a set of $\overline{k}$-points of $\PP^1$.  Let $g=\gamma\circ f \circ \gamma^{-1}$, so $C_g = \gamma(C_f)$.  If we extend $g$ to a model over some non-empty open set $U\subset X$, then, the critical locus of the induced map on fibers over any $\overline{k}$ point of $U$ has constant critical locus $C_g$. But, by Proposition \ref{prop:rigidity} there are only finitely many maps with critical locus $C_f$. Therefore $g$ must be defined over $\overline{k}$. 

\end{proof}

\begin{rem}
In \cite{PetscheSzpiroTepper} it is proved that a rational map with everywhere potential good reduction is isotrivial. We could have used this fact, to prove the preceding lemma in the case that $\Char K$ is zero.
\end{rem}

\begin{proof}[Proof of Theorem \ref{thm:Main}]
 Let $F=[f]_K$ be an element of $\Sh(d,K,S)$. Then $f$ is admissible, whence, by Lemma \ref{lem:nonisotrivial}, $C_f$ is an admissible reduced effective divisor. $C_f$ has degree at most $4d-4$ and splitting degree over $K$ at most $(2d-2)!$, since the degree of $R_f$ is at most $2d-2$ and $B_f=f(R_f)$ with $f$ defined over $K$. Moreover, $C_f$ has good reduction outside of $S$. So $C_f$ is an element of $\Sh_\mathrm{div}^\star(4d-4,(2d-2)!, K,S)$. By Theorem \ref{thm:finiteness}, $C_f$ belongs to one of at most finitely many $K$-equivalence classes. Let $\mcC$ be a set of representatives for such equivalence classes and let $Y=\cup_{C\in\mcC} \overline{C}$. Thus, for every $F$ in $\Sh(d,K,S)$, we can write $F=[f]_K$ for some rational map $f/K$ for which $C_f\subset Y$. Now the finiteness of $\Sh(d,K,S)$ follows from Proposition \ref{prop:rigidity}.    
\end{proof}

\bibliography{shafarevich_bib}

\begin{bibdiv}
\begin{biblist}

\bib{Arakelov}{article}{
      author={Arakelov, S.~Ju.},
       title={Families of algebraic curves with fixed degeneracies},
        date={1971},
        ISSN={0373-2436},
     journal={Izv. Akad. Nauk SSSR Ser. Mat.},
      volume={35},
       pages={1269\ndash 1293},
      review={\MR{0321933}},
}

\bib{EllipticShaf}{article}{
      author={Bandini, Andrea},
      author={Longhi, Ignazio},
      author={Vigni, Stefano},
       title={Torsion points on elliptic curves over function fields and a
  theorem of {I}gusa},
        date={2009},
        ISSN={0723-0869},
     journal={Expo. Math.},
      volume={27},
      number={3},
       pages={175\ndash 209},
         url={http://dx.doi.org/10.1016/j.exmath.2008.06.001},
      review={\MR{2555367}},
}

\bib{Canci}{article}{
      author={Canci, Jung~Kyu},
      author={Peruginelli, Giulio},
      author={Tossici, Dajano},
       title={On some notions of good reduction for endomorphisms of the
  projective line},
        date={2013},
        ISSN={0025-2611},
     journal={Manuscripta Math.},
      volume={141},
      number={1-2},
       pages={315\ndash 331},
         url={http://dx.doi.org.ezproxy.gc.cuny.edu/10.1007/s00229-012-0573-y},
      review={\MR{3042691}},
}

\bib{EG}{article}{
      author={Evertse, J.-H.},
      author={Gy{\H{o}}ry, K.},
       title={On the number of polynomials and integral elements of given
  discriminant},
        date={1988},
        ISSN={0236-5294},
     journal={Acta Math. Hungar.},
      volume={51},
      number={3-4},
       pages={341\ndash 362},
         url={http://dx.doi.org.ezproxy.gc.cuny.edu/10.1007/BF01903342},
      review={\MR{956987 (89k:12006)}},
}

\bib{EGbook}{book}{
      author={Evertse, Jan-Hendrik},
      author={Gy\H{o}ry, K\'alm\'an},
       title={Unit equations in {D}iophantine number theory},
      series={Cambridge Studies in Advanced Mathematics},
   publisher={Cambridge University Press, Cambridge},
        date={2015},
      volume={146},
        ISBN={978-1-107-09760-5},
  url={http://dx.doi.org.proxy.its.virginia.edu/10.1017/CBO9781316160749},
      review={\MR{3524535}},
}

\bib{Faltings}{article}{
      author={Faltings, G.},
       title={Endlichkeitss\"atze f\"ur abelsche {V}ariet\"aten \"uber
  {Z}ahlk\"orpern},
        date={1983},
        ISSN={0020-9910},
     journal={Invent. Math.},
      volume={73},
      number={3},
       pages={349\ndash 366},
         url={http://dx.doi.org/10.1007/BF01388432},
      review={\MR{718935}},
}

\bib{smallPi1}{article}{
      author={Harada, Shinya},
      author={Hiranouchi, Toshiro},
       title={Smallness of fundamental groups for arithmetic schemes},
        date={2009},
        ISSN={0022-314X},
     journal={J. Number Theory},
      volume={129},
      number={11},
       pages={2702\ndash 2712},
  url={http://dx.doi.org.proxy.its.virginia.edu/10.1016/j.jnt.2009.03.010},
      review={\MR{2549526}},
}

\bib{MilnorBook}{book}{
      author={Milnor, John},
       title={Dynamics in one complex variable},
     edition={Third},
      series={Annals of Mathematics Studies},
   publisher={Princeton University Press, Princeton, NJ},
        date={2006},
      volume={160},
        ISBN={978-0-691-12488-9; 0-691-12488-4},
      review={\MR{2193309}},
}

\bib{Mori}{article}{
      author={Mori, Shigefumi},
       title={Projective manifolds with ample tangent bundles},
        date={1979},
        ISSN={0003-486X},
     journal={Ann. of Math. (2)},
      volume={110},
      number={3},
       pages={593\ndash 606},
         url={http://dx.doi.org/10.2307/1971241},
      review={\MR{554387}},
}

\bib{Petsche}{article}{
      author={Petsche, Clayton},
       title={Critically separable rational maps in families},
        date={2012},
        ISSN={0010-437X},
     journal={Compos. Math.},
      volume={148},
      number={6},
       pages={1880\ndash 1896},
         url={http://dx.doi.org/10.1112/S0010437X12000346},
      review={\MR{2999309}},
}

\bib{Parshin}{article}{
      author={Par\v~sin, A.~N.},
       title={Algebraic curves over function fields. {I}},
        date={1968},
        ISSN={0373-2436},
     journal={Izv. Akad. Nauk SSSR Ser. Mat.},
      volume={32},
       pages={1191\ndash 1219},
      review={\MR{0257086}},
}

\bib{PetscheStout}{article}{
      author={Petsche, Clayton},
      author={Stout, Brian},
       title={On quadratic rational maps with prescribed good reduction},
        date={2015},
        ISSN={0002-9939},
     journal={Proc. Amer. Math. Soc.},
      volume={143},
      number={3},
       pages={1145\ndash 1158},
         url={http://dx.doi.org/10.1090/S0002-9939-2014-12291-X},
      review={\MR{3293730}},
}

\bib{PetscheSzpiroTepper}{article}{
      author={Petsche, Clayton},
      author={Szpiro, Lucien},
      author={Tepper, Michael},
       title={Isotriviality is equivalent to potential good reduction for
  endomorphisms of {$\Bbb P^N$} over function fields},
        date={2009},
        ISSN={0021-8693},
     journal={J. Algebra},
      volume={322},
      number={9},
       pages={3345\ndash 3365},
  url={http://dx.doi.org.ezproxy.gc.cuny.edu/10.1016/j.jalgebra.2008.11.027},
      review={\MR{2567424 (2011e:14023)}},
}

\bib{Serre}{incollection}{
      author={Serre, J.-P.},
       title={Modules projectifs et espaces fibr\'es \`a fibre vectorielle},
        date={1958},
   booktitle={S\'eminaire {P}. {D}ubreil, {M}.-{L}. {D}ubreil-{J}acotin et {C}.
  {P}isot, 1957/58, {F}asc. 2, {E}xpos\'e 23},
   publisher={Secr\'etariat math\'ematique, Paris},
       pages={18},
      review={\MR{0177011}},
}

\bib{Shafarevich}{incollection}{
      author={Shafarevich, I.~R.},
       title={Algebraic number fields},
        date={1963},
   booktitle={Proc. {I}nternat. {C}ongr. {M}athematicians ({S}tockholm, 1962)},
   publisher={Inst. Mittag-Leffler, Djursholm},
       pages={163\ndash 176},
      review={\MR{0202709}},
}

\bib{SilvermanSpace}{article}{
      author={Silverman, Joseph~H.},
       title={The space of rational maps on {$\mathbb{P}^1$}},
        date={1998},
        ISSN={0012-7094},
     journal={Duke Math. J.},
      volume={94},
      number={1},
       pages={41\ndash 77},
  url={http://dx.doi.org.ezproxy.gc.cuny.edu/10.1215/S0012-7094-98-09404-2},
      review={\MR{1635900 (2000m:14010)}},
}

\bib{SilvermanBook}{book}{
      author={Silverman, Joseph~H.},
       title={The arithmetic of dynamical systems},
      series={Graduate Texts in Mathematics},
   publisher={Springer},
     address={New York},
        date={2007},
      volume={241},
        ISBN={978-0-387-69903-5},
         url={http://dx.doi.org.ezproxy.gc.cuny.edu/10.1007/978-0-387-69904-2},
      review={\MR{2316407 (2008c:11002)}},
}

\bib{SilvermanMod}{book}{
      author={Silverman, Joseph~H.},
       title={Moduli spaces and arithmetic dynamics},
      series={CRM Monograph Series},
   publisher={American Mathematical Society},
     address={Providence, RI},
        date={2012},
      volume={30},
        ISBN={978-0-8218-7582-7},
      review={\MR{2884382}},
}

\bib{SilvermanShaf}{misc}{
      author={Silverman, Joseph~H.},
       title={Good reduction and {S}hafarevich-type theorems for dynamical
  systems with portrait level structures},
        date={2017},
        note={arXiv:1703.00823},
}

\bib{SzpiroToronto}{article}{
      author={Szpiro, Lucien},
       title={Proof of the {S}hafarevich conjecture for self maps of
  $\mathbb{P}^1$ ({B}rown 2015, {T}oronto 2017)},
  journal={\url{http://www.fields.utoronto.ca/talks/Proof-Shafarevich-conjecture-self-maps-mathbbP1}},
        note={Video of conference talk},
}

\bib{SzpiroSem}{incollection}{
      author={Szpiro, Lucien},
       title={Propri\'et\'es num\'eriques du faisceau dualisant r\'elatif},
        date={1981},
   booktitle={Séminaire sur les pinceaux de courbes de genre au moins deux},
      series={Ast\'erisque},
      volume={86},
   publisher={SMF},
     address={Paris},
       pages={44\ndash 78},
}

\bib{ST}{article}{
      author={Szpiro, Lucien},
      author={Tucker, Thomas~J.},
       title={A {S}hafarevich-{F}altings theorem for rational functions},
        date={2008},
        ISSN={1558-8599},
     journal={Pure Appl. Math. Q.},
      volume={4},
      number={3, part 2},
       pages={715\ndash 728},
      review={\MR{MR2435841 (2009g:14027)}},
}

\bib{West}{misc}{
      author={West, Lloyd~W.},
       title={The moduli space of cubic rational maps},
        date={2014},
        note={arXiv:1408.3247},
}

\end{biblist}
\end{bibdiv}

\end{document}